\documentclass[a4paper,11pt]{article}

\input{styleart.sty}
\usepackage{xcolor}

\title{First-order sentences in random groups II: $\forall\exists$-sentences.}

\author{O. Kharlampovich, R. Sklinos}

\begin{document}

\maketitle

\begin{abstract}
We prove that a random group, in Gromov's density model with $d<1/16$   satisfies with overwhelming probability a universal-existential  first-order sentence $\sigma$ (in the language of groups) if and only if $\sigma$ is true in a nonabelian free group.
\end{abstract}

\section{Introduction}
In this paper we continue our work that connects random groups with the first-order theory of nonabelian free groups (see \cite{KhS}). We generalize our previous result, that a random group (of density $d<1/16$) satisfies with overwhelming probability a universal sentence in the language of groups if and only if the sentence is satisfied in a nonabelian free group, to $\forall\exists$-sentences. Our main result is 

\begin{thmIntro}
Let $\sigma$ be a $\forall\exists$ first-order sentence in the language of groups. Let $0\leq d<1/16$ be a real number. Then a random group of density $d$ satisfies, with overwhelming probability, the sentence $\sigma$ if and only if a non abelian free group satisfies $\sigma$.
\end{thmIntro}

We will make heavy use of the machinery developed for answering Tarski's question and in particular {\em formal solutions, towers, closures of towers} (see \cite{MR1972179}, \cite{MR2154989}) and the {\em process of validating a $\forall\exists$-sentence} (see \cite{Sela4}, \cite{MR2293770}).

\section{Preliminaries}

\subsection{The density model}
Recall Gromov's density model of randomness. 

\begin{definition}[Gromov's Density Model]\label{Density}
Let $\mathbb{F}_n:=\langle e_1, \ldots, e_n\rangle$ be a free group of rank $n$. Let $S_{\ell}$ be the set of reduced words on $e_1, \ldots, e_n$ of length $\ell$. 

Let $0\leq d\leq 1$. Then a random set of relators of density $d$ at length $\ell$ is a subset of $S_{\ell}$ that consists of  $(2n-1)^{d\ell}$-many elements picked randomly (uniformly and independently) among all elements of $S_{\ell}$. 

A group $G:=\langle e_1,\ldots, e_n \ | \ \mathcal{R} \ \rangle$ is called random of density $d$ at length $l$ if $R$ is a random set of relators of density $d$ at length $\ell$.

A random group of density $d$ satisfies some property (of presentations) $P$ with overwhelming probability (w.o.p.), if the probability of occurrence of $P$ tends to $1$ as $\ell$ goes to infinity. 
\end{definition}

We note in passing that at density $0$, we, formally, have one relator of length $\ell$, but the usual convention is that instead of one we have a finite fixed number of relators of length $\leq \ell$ (see \cite[Remark 12]{MR2205306}). Hence the few-relator model (see \cite[Definition 1]{MR2205306}) is a special case of the density model for $d=0$. 

Heuristically, one can understand this as follows: the ratio of groups in $(2n-1)^{d\ell}$ relators all of length $\ell$ that satisfy property $P$ over all such groups is a number, say $p$, that for ``interesting properties" will depend on $\ell$, i.e. $p:=p(\ell)$. If $p(\ell)$ goes to $1$ as $\ell$ goes to $\infty$, then we say that w.o.p. a random group has property $P$.

We will need the following results from \cite{KhS}. 

\begin{theorem}\label{univ1} \cite[Theorem 7.14]{KhS}
Let $d<1/16$. Let $\sigma$ be a universal sentence in the language of groups. Then $\sigma$ is almost surely true in the random group of density $d$ if and only if it is true in a nonabelian free group.
\end{theorem}

Theorem $\ref{univ1}$ was obtained as a corollary of the following.

\begin{theorem}\label{univ} \cite[Proposition 7.13]{KhS}
Let $d<1/16$. Let $V(\bar{x})=1$ be system of equations. Suppose $\Gamma_\ell$ is a random group of density $d$ at length $\ell$ and $\pi_\ell:\F_n\rightarrow\Gamma_\ell$ the canonical quotient map. 

Then, every solution $V(\bar b_\ell)=1$ in $\Gamma_\ell$ is the image of a solution $V(\bar{c}_\ell)=1$ in $\F_n$, under the canonical quotient map $\pi_{\ell}$, i.e.   $\pi_\ell(\bar{c}_\ell)=\bar{b}_\ell$, with probability tending to $1$ as $\ell$ goes to infinity.  
\end{theorem}

To understand Theorem \ref{univ} in the light of Definition \ref{Density} one considers the following property $P$: for a fixed $V(\bar x)=1$, every solution of $V(\bar x)=1$, is the pre-image of a solution of $V(\bar x)=1$ in $\F$, under the canonical quotient map. 

Alternatively, under the same interpretation, one can think that the above theorem says that either $\bar b_\ell$ is the image (under the canonical map) of a solution of $V(\bar x)=1$ in $\F_n$, or the probability that $V(\bar{b}_\ell)=1$ tends to $0$ as $\ell$ goes to $\infty$. 

With little care about defining constants (coefficients) in random groups one can make sense of a first-order sentence with constants.

\begin{definition}
Let $\Gamma_\ell$ be a random group of density $d$ at length $\ell$. Let $b_\ell$ be an element in $\Gamma_\ell$. Then $b_\ell$ is a constant if there exists an element $b\in\F_n$ such that $b_\ell$ is the image of $b$, under the canonical quotient map, for all $\ell\in\mathbb{N}$.  
\end{definition}


Hence, under the above definition, a sentence with constants in $\F_n$ makes sense in a random group and likewise a sentence with constants in a random group makes sense in $\F_n$. 

For the purposes of this paper we need a generalization of Theorem \ref{univ1} to first-order sentences defined over constants.  

\begin{theorem}\label{UnivwithCon1}
Let $d<1/16$. Let $\sigma$ be a universal sentence in the language of groups over $\F_n$ with constants. Then, $\sigma$ is true in $\F_n$ if and only if it is true with overwhelming probability in a random group of density $d$.
\end{theorem}

It is essentially a corollary of a generalization of Theorem \ref{univ} with constants. 

\begin{theorem}\label{UnivwithCon}
Let $d<1/16$. Let $V(\bar{x},\bar a)=1$ be system of equations over $\F_n$. Suppose $\Gamma_\ell$ is a random group of density $d$ at length $\ell$ and $\pi_\ell:\F_n\rightarrow\Gamma_\ell$ the canonical quotient map. 

Then, every solution $V(\bar b_\ell,\bar a)=1$ in $\Gamma_\ell$ is the image of a solution $V(\bar{c}_\ell,\bar a)=1$ in $\F_n$, under the canonical quotient map $\pi_{\ell}$, i.e.   $\pi_\ell(\bar{c}_\ell)=\bar{b}_\ell$, with probability tending to $1$ as $\ell$ goes to infinity. 
\end{theorem}
%

\subsection{Boolean combinations of Universal Existential Axioms}

The following lemma is due to Malcev. Its proof uses the fact that all solutions in a free group of the equation  $x^2y^2z^2=1$ commute.
\begin{lemma}\label{Conj}
Let $\F:=\langle e_1, e_2, \ldots\rangle$ be a nonabelian free group. Then, conjunctions of equations are equivalent (over constants), in $\F$, with one equation:  
$$\F\models \forall x, y \bigl((x=1\land y=1)\leftrightarrow (x^2e_1)^2e_1^{-2}=((ye_2)^2e_2^{-2})^2\bigr)$$
\end{lemma}
For disjunctions we get the following
\begin{lemma}\label{Disj}
Let $\F:=\langle e_1, e_2, \ldots\rangle$ be a nonabelian free group. Then, a disjunction of equations is equivalent (over constants), in $\F$, with four conjunctions of equations:
$$\F\models \forall x, y \bigl((x=1\lor y=1)\leftrightarrow \bigwedge_{a\in\{e_1, e_1^{-1}\}\atop b\in\{e_2, e_2^{-1}\}}[x^{a}, y^{b}]=1\bigr)$$
\end{lemma}

In particular, one easily obtains the following corollary. 

\begin{corollary} \label{one_eq_free}
A  disjunction of conjunctions (or conjunction of disjunctions) of equations over $\F$ is equivalent, in $\mathbb F$, to one equation.
\end{corollary}

Since the above Corollary can be expressed by a universal formula we also get.   

\begin{corollary}\label{one_eq} 
A disjunction of conjunctions (or conjunction of disjunctions) of equations over $\Gamma$ is equivalent in $\Gamma$ to the same one equation as in $\mathbb F$.
\end{corollary}
 
\begin{lemma}(cf. \cite[Lemma 6]{MR2154989}) \label{reduction}
Let $\tau$ be a $\forall\exists$ first-order sentence in the language of groups. Then, $\tau$ is equivalent in $\mathbb F_n$ to a sentence $\zeta$ of the form
 $$\forall\bar{x}\exists\bar{y} \bigl(\sigma(\bar{x},\bar{y}, \bar{a})=1 \land \psi(\bar{x},\bar{y}, \bar{a})\neq 1\bigr)$$  
where $\sigma(\bar{x},\bar{y}, \bar{a})=1$ is an equation and $\psi(\bar{x},\bar{y}, \bar{a})\neq 1$ is an inequation, both over constants from $\F_n$.

Moreover, if $d<1/16$, then $\tau$ is almost surely true in a random group of density $d$ if and only if the sentence $\zeta$ is.   

\end{lemma}
\begin{proof} Every $\forall\exists$ sentence in the language of groups is (logically) equivalent to a formula in prenex (disjunctive) normal form 
$$\forall\bar x\exists \bar y\Bigl( \bigvee_{i=1}^m\bigl(\Sigma_i(\bar{x},\bar{y})=1 \land \Psi_i(\bar{x},\bar{y})\neq 1\bigr)\Bigr)$$
In any non-trivial group the quantifier-free part, $\bigvee_{i=1}^m\bigl(\Sigma_i(\bar{x},\bar{y})=1 \land \Psi_i(\bar{x},\bar{y})\neq 1\bigr)$, of the above sentence is equivalent to 
$$\exists \bar z_1,\ldots ,\bar z_m \bigl((\bigwedge_{i=1}^m \bar z_i\neq 1)\bigwedge\bigvee_{i=1}^m\bigl(\Sigma_i(\bar{x},\bar{y})=1 \land \Psi_i(\bar{x},\bar{y})= \bar z_i\bigr)$$ 
By Corollary \ref{one_eq_free}, the disjunction of conjunctions of equations 
$$\bigvee_{i=1}^m\bigl(\Sigma_i(\bar{x},\bar{y})=1 \land \Psi_i(\bar{x},\bar{y})=\bar z_i\bigr)$$
is equivalent to one equation $\sigma (\bar{x},\bar{y},\bar z_1,\ldots ,\bar z_m,\bar a)=1$ over constants in $\mathbb F_n$. Similarly, the conjunction $\bigwedge_{i=1}^m \bar z_i\neq 1$ is equivalent to a single inequation $\psi(\bar z_1,\ldots, \bar z_m,\bar a)\neq 1$ over constants in $\F_n$.  Hence, we can take for $\zeta$ the following sentence

$$\forall\bar x\exists \bar y \exists z_1, \ldots, z_m \bigl(\psi(\bar z_1, \ldots, \bar z_m,\bar a)\neq 1\bigwedge \sigma (\bar{x},\bar{y},\bar z_1,\ldots, \bar z_m,\bar a)=1 \bigr)$$

For a random group $\Gamma$ of density $d<1/16$ we argue as follows. The sentence $\forall\bar x\exists \bar y\Bigl( \bigvee_{i=1}^m\bigl($ $\Sigma_i(\bar{x},\bar{y})=1 \land \Psi_i(\bar{x},\bar{y})\neq 1\bigr)\Bigr)$ is almost surely true in $\Gamma$ if and only if $\forall \bar x \exists \bar y \exists \bar z_1,\ldots,\bar z_m \Bigl((\bigwedge_{i=1}^m \bar z_i\neq 1)\bigwedge\bigvee_{i=1}^m\bigl(\Sigma_i(\bar{x},\bar{y})=1 \land \Psi_i(\bar{x},\bar{y})= \bar z_i\bigr)\Bigr)$ is almost surely true in $\Gamma$. In addition, 
$$\F_n\models \forall\bar x\forall\bar y\forall\bar z\Bigl( \Big[(\bigwedge_{i=1}^m \bar z_i\neq 1)\bigwedge\bigvee_{i=1}^m\bigl(\Sigma_i(\bar{x},\bar{y})=1 \land \Psi_i(\bar{x},\bar{y})= \bar z_i\bigr)\Big]\leftrightarrow $$ $$\bigl(\psi(\bar z_1, \ldots, \bar z_m,\bar a)\neq 1\bigwedge \sigma (\bar{x},\bar{y},\bar z_1,\ldots,\bar z_m,\bar a)=1\bigr)\Bigr)$$

The above sentence, by Theorem \ref{UnivwithCon1}, is almost surely true in $\Gamma$. In particular, $\forall \bar x \exists \bar y \exists \bar z_1,\ldots, $ $\bar z_m\Bigl((\bigwedge_{i=1}^m z_i\neq 1)\bigwedge\bigvee_{i=1}^m\bigl(\Sigma_i(\bar{x},\bar{y})=1 \land \Psi_i(\bar{x},\bar{y})= \bar z_i\bigr)\Bigr)$ is almost surely true in $\Gamma$ if and only if $\forall\bar x\exists \bar y \exists \bar z_1, \ldots, \bar z_m \bigl(\psi(\bar z_1, \ldots, \bar z_m,\bar a)\neq 1\bigwedge \sigma (\bar{x},\bar{y},\bar z_1,\ldots,\bar z_m,\bar a)=1 \bigr)$ is almost surely true in $\Gamma$.
\end{proof}

\subsection {Validation of a $\forall\exists$ sentence in nonabelian free groups.}\label{Validation}

Let $\forall\bar{x}\exists\bar{y} \bigl(\Sigma(\bar{x},\bar{y}, \bar{a})=1 \land \Psi(\bar{x},\bar{y}, \bar{a})\neq 1\bigr)$,  
where $\Sigma(\bar{x},\bar{y}, \bar{a})=1$ is a conjuction of equations and $\Psi(\bar{x},\bar{y}, \bar{a})\neq 1$ a conjuction of inequations, be a true sentence in a nonabelian free group $\F:=\langle \bar a \rangle$. The idea, for validating the above sentence, is to find witnesses for the existentially quantified variables $\bar{y}$ in terms of the universally quantified variables $\bar{x}$ and the constants $\bar{a}$ as words in $\langle\bar{x}, \bar{a}\rangle$. Indeed, the first step of the validating process is based on the following theorem \cite[Theorem 1.2]{MR1972179}, \cite{MR2154989}:

\begin{theorem}\label{Merz}
Let $\mathbb{F}\models\forall\bar{x}\exists\bar{y}(\Sigma(\bar{x},\bar{y}, \bar{a})=1 \land \Psi(\bar{x},\bar{y}, \bar{a})\neq 1)$. Then, there exists a tuple of words $\bar{w}(\bar{x}, \bar{a})$ in the free group $\langle\bar{x}, \bar{a}\rangle$, such that $\Sigma(\bar{x},\bar{w}(\bar{x}, \bar{a}), \bar{a})$ is trivial in $\langle\bar{x}, \bar{a}\rangle$ and moreover $\mathbb{F}\models \exists\bar{x}\Psi(\bar{x},\bar{w}(\bar{x}, \bar{a}), \bar{a})\neq 1$. 
\end{theorem}

In the special case where no inequations exist, Theorem \ref{Merz} is known as {\em Merzlyakov's theorem} and leads to the equality of the positive theories of nonabelian free groups. We think of  $\bar{w}(\bar{x},\bar{a})$ as validating the sentence in a particular subset of $\mathbb{F}^{\abs{\bar{x}}}$. What is left to do is find validating witnesses for the complement of this subset. The subset of $\mathbb{F}^{\abs{\bar{x}}}$ for which the formal solution {\bf does not} work is first-order definable by the union of the following ``varieties" $\psi_1(\bar{x},\bar{w}(\bar{x}, \bar{a}), \bar{a})=1, \ldots, \psi_k(\bar{x},\bar{w}(\bar{x}, \bar{a}),\bar{a})=1 $, where each $\psi_i(\bar{x},\bar{y}, \bar{a})$, for $i\leq k$, is a word in  $\Psi(\bar{x},\bar{y},\bar{a})$. One can further split each variety $\psi_i(\bar{x},\bar{w}(\bar{x}, \bar{a}))$ in finitely many irreducible varieties, i.e. systems of equations $\Sigma_{i1}(\bar{x},\bar{a})=1, \ldots, \Sigma_{im_i}(\bar{x},\bar{a})=1$  for $i\leq k$, such that $L_{ij}:=\langle \bar{x},\bar{a} \ | \ \Sigma_{ij}(\bar{x}, \bar{a})\rangle$, for $i\leq k$ and $j\leq m_i$, is a (restricted) limit group. 

The iterative step of the process uses a further generalization of Merzlyakov's theorem that we record next (see \cite[Theorem 1.18]{MR1972179}, \cite{MR2154989}). For convenience of notation we denote by $G(\bar{x})$ a group $G$ with generating set $\bar{x}$.

\begin{theorem}\label{MerzTowers}
Let $L(\bar{x},\bar{a}):=\langle \bar{x}, \bar{a} \ | \ R(\bar{x},\bar{a})\rangle$ be a restricted limit group, and $T(\bar{x},\bar{z}, \bar{a})$ a tower constructed from a well-structured resolution of $L(\bar x, \bar a)$.  

Suppose 
$$\mathbb{F}\models \forall\bar{x} (R(\bar{x},\bar{a})=1 \rightarrow \exists \bar{y}(\Sigma(\bar{x},\bar{y},\bar{a})=1 \land \Psi(\bar{x},\bar{y},\bar{a})\neq 1))$$
Then there exist,  $C_1(\bar{x},\bar{z},\bar{s},\bar{a})$, \ldots, $C_q(\bar{x},\bar{z},\bar{s},\bar{a})$ a covering closure of $T(\bar{x},\bar{z},\bar{a})$ and ``formal solutions",  $\bar{w}_1(\bar{x},\bar{z},\bar{s},\bar{a}), \ldots, \bar{w}_q(\bar{x},\bar{z},\bar{s},\bar{a})$ with the following properties: 
\begin{itemize}
\item for each $1\leq i \leq q$, the words $\Sigma(\bar{x}, \bar{w}_i(\bar{x},\bar{z},\bar{s},\bar a),\bar a)$ are trivial in $C_i(\bar{x},\bar{z},\bar{s},\bar a)$. 
\item for each $1\leq i\leq q$, there exists a morphism $h_i:C_i\rightarrow \mathbb{F}$, which is the identity on $\F$ and such that $\Psi(h_i(\bar{x}), h_i(\bar{w}_i),\bar{a})\neq 1$. 
\end{itemize}
\end{theorem}

In principle formal solutions do not exist in arbitrary limit groups, but only in limit groups that admit a special structure - the structure of a tower. A tower is constructed as a nested graph of groups based on a free group and ``gluing" at each step either a surface with boundaries (along the boundaries) or a free abelian group along a direct factor. The covering closure of a tower is a finite set of towers where the free abelian ``floors" are augmented to finite index supergroups (see \cite[Definition 1.15]{MR1972179}). It covers the original tower in the sense that every morphism from the original tower to a free group, extends (as a closure contains the tower as a subgroup) to a morphism from some closure. The precise construction is of no practical importance for this paper, but we note that as a consequence the (new) subset of $\mathbb{F}^{\abs{x}}$ for which the (new) formal solution {\bf does not work} is not a union of varieties anymore, but is {\bf contained} in a union of Diophantine sets. This subtlety makes the proof that the process terminates quite involved. In any case, the important fact for us is that after repeatedly using the above theorem we will eventually cover all of $\mathbb{F}^{\abs{x}}$ with finitely many subsets for which some formal solution works.

We will give some more details on the procedure avoiding the technical results that imply its termination after finitely many steps. It will be important to carefully collect the ``bad" tuples in a set, in such a way that the procedure terminates. This set will occasionally be larger than needed, i.e. it will also contain tuples for which a formal solution already works, but this is unavoidable, under the existing methods, if one wants to guarantee the termination. We next explain how this works. For presentational purposes we will first record a special case, called {\em the minimal rank case}.
\ \\ \\
{\bf The minimal rank case.} The assumption in this case is that all (restricted) limit groups $L_{ij}$, for $i\leq k$ and $j\leq m_i$, that collect the ``bad" tuples in the first step of the procedure {\bf do not admit} a surjection (that is the identity on $\F$) to a free product $\F*\F_n$, for some nontrivial free group $\F_n$. This simplified version of the iterative procedure is presented, for example,  in \cite[Section 1]{Sela4}. This assumption considerably simplifies the technicalities in the construction of the towers on each consecutive step and guarantees that their complexity decreases. 

\begin{enumerate}
\item In the first step of the procedure we apply Theorem \ref{Merz} to the sentence $\forall\bar{x}\exists\bar{y}(\Sigma(\bar{x},\bar{y}, \bar{a})=1 \land \Psi(\bar{x},\bar{y}, \bar{a})\neq 1)$ and obtain a formal solution $\bar{w}(\bar{x},\bar{a})$ that {\bf does not work} in the union of varieties $\psi_1(\bar{x},\bar{w}(\bar{x},\bar{a}),\bar{a})=1, \ldots, \psi_k(\bar{x},\bar{w}(\bar{x}, \bar{a}),\bar{a})=1$.
\item The latter union of varieties can be further decomposed as $Hom_{\F}(L_{ij},\F)$, for $i\leq k$ and $j\leq m_i$, where $L_{ij}$ is a restricted limit group, and $Hom_{\F}(L_{ij},\F)$ is the set of restricted homomorphisms from $L_{ij}$ to $\F$. Equivalently, this can be seen as a decomposition of each variety to its irreducible components. 
\item In the iterative step of the procedure we work with each $L_{ij}$ in parallel. Thus, we can fix  $L:=L_{ij}$ for some $i\leq k$ and $j\leq m_i$. To each (restricted) limit group, $L$, we can assign finitely may towers (based on well-structured resolutions of $L$), $T_1, \ldots, T_{\ell}$, such that any (restricted) morphism $h:L\rightarrow\F$ factors through $T_i$, for some $i\leq\ell$. We will, again, work in parallel with each $T_i$. Thus, we can fix $T:=T_i$, for some $i\leq \ell$.
\item  We apply Theorem \ref{MerzTowers} for the couple $L$ and $T$ and obtain finitely many closures $C_1(\bar x, \bar z, \bar s, \bar a), \ldots C_q(\bar x, \bar z, \bar s, \bar a)$ and formal solutions $\bar{w}_1(\bar x, \bar z, \bar s, \bar a), \ldots, \bar{w}_q(\bar x, \bar z, \bar s, \bar a)$ for each closure, that each {\bf does not work} in the definable set 
$$\exists \bar z,\bar s\bigl(\Sigma_{C_i}(\bar{x}, \bar{z}, \bar{s}, \bar{a})=1\bigr)\land \forall \bar z, \bar s \bigl(\psi_1(\bar x, \bar w_i(\bar x, \bar z, \bar s, \bar a) , \bar a)=1\lor\ldots\lor\psi_k(\bar x, \bar w_i(\bar x, \bar z, \bar s, \bar a) , \bar a)=1)\bigr)$$
\item We do not continue with the previous definable set, but rather with the larger set defined by 
$$\exists \bar z, \bar s \biggl(\Sigma_{C_i}(\bar{x}, \bar{z}, \bar{s}, \bar{a})=1\bigwedge\bigl(\psi_1(\bar x, \bar w_i(\bar x, \bar z, \bar s, \bar a) , \bar a)=1\lor\ldots\lor\psi_k(\bar x, \bar w_i(\bar x, \bar z, \bar s, \bar a) , \bar a)=1\bigr)\biggr)$$
\item We work with each closure in parallel. Thus we fix $C(\bar x, \bar z, \bar s, \bar a):=C_i(\bar x, \bar z, \bar s, \bar a)$ for some $i\leq q$. We consider the set of (restricted) morphisms, $Hom_{\F}(C,\F)$, whose images satisfy $\psi_1(\bar x, \bar w_i(\bar x, \bar z, \bar s, \bar a) , \bar a)=1\lor\ldots\lor\psi_k(\bar x, \bar w_i(\bar x, \bar z, \bar s, \bar a) , \bar a)=1$ in $\F$. There exist finitely many (restricted) limit groups, $QL_1, \ldots, QL_m$,   which are quotients of $C$ and such that a (restricted) morphism $h:C\rightarrow\F$ satisfies the previous condition if and only if it factors through one of this quotients. 
\item We repeat the procedure for each $QL_i$ and towers based on resolutions constructed as in \cite[Proposition 1.13]{Sela4}, by Theorem 1.18 in \cite{Sela4} the procedure terminates after finitely many steps.
\end{enumerate}
 
\ \\ \\
We next record the general case.  
\ \\ \\
{\bf The general case.} In the general case not only the resolutions (used to construct towers) are modified but also the family of morphisms that descend to the next step of the procedure. Instead of well-structured resolutions we restrict to a special subclass called {\em well-separated resolutions} (see \cite[Definition 2.2]{Sela4}). One of the properties of well-separated resolutions is that they can be used to endow each surface that corresponds to a $QH$ subgroup of every $JSJ$ decomposition along the resolution with a family of (two-sided, disjoint, non null-homotopic, non parallel) simple closed curves. This family of curves induces a splitting, as a graph of groups, of the fundamental group of the surface by cutting it along them.    

\begin{definition} 
Let $\Sigma$ be a compact surface. Given a homomorphism $h:\pi _1(\Sigma)\rightarrow H$
a family of pinched curves is a collection $\mathcal C$ of disjoint, non-parallel, two-sided
simple closed curves, none of which is null-homotopic, such that the fundamental
group of each curve is contained in $ker h$ (the curves may be parallel to a boundary component).

The map $\phi$  is {\em non-pinching} \index{Non-pinching map} if there is no pinched curve: $\phi$ is injective in restriction to
the fundamental group of any simple closed curve which is not null-homotopic.
\end{definition}

If $\eta_{i+1}:L_i\rightarrow L_{i+1}$ is part of a well-separated resolution and $\pi_1(\Sigma):=Q$ is a $QH$ subgroup in a $JSJ$ decomposition of a freely indecomposable free factor of $L_i$, then to $\eta_{i+1}$ and the free decomposition of $L_{i+1}$ corresponds a family of pinched curves on $\Sigma$ that we will use in order to define a special class of morphisms called {\em taut with respect to the given resolution} (see \cite[Definition 2.4]{Sela4}). 

\begin{remark}
When a morphism $h:L\rightarrow\F$ factors through a resolution, $L:=L_0\rightarrow L_1\rightarrow L_2\rightarrow \ldots \rightarrow L_p$ then for each $i\leq p$, there exists a section with respect to the quotient map $\eta_i:L_{i-1}\rightarrow L_i$, for $0<i\leq p$, i.e. a morphism, $h_i:L_i\rightarrow\F$, such that $h=h_i\circ\eta_i\circ\alpha_{i-1}\circ\ldots\circ\eta_1\circ\alpha_0$, where $\alpha_i$ is in $Mod(L_i)$. 

For economy of notation when we say that $h$ (does not) kill an element of $L_i$ we mean that $h_i$ (does not) kill it.
\end{remark}

\begin{definition}[Taut morphism]\label{Taut}
Let $h$ be a morphism that factors through a well-separated resolution 
$$ L:=L_0\rightarrow L_1\rightarrow L_2\rightarrow \ldots \rightarrow L_p$$

Then $h$ is taut with respect to the given resolution if the following conditions hold:
\begin{itemize}
\item every abelian edge group in each abelian decomposition at every level of the resolution is not killed by $h$.
\item every $QH$ subgroup in each abelian decomposition at every level of the  resolution does not have an abelian image under $h$.
\item for any vertex group of the graph of groups decomposition of a $QH$ subgroup of any $JSJ$ decomposition in any level of the resolution, obtained by the family of pinched curves, the morphism, $h\circ \alpha$, i.e. $h$ pre-composed by any modular automorphism of the $QH$ subgroup, restricted to this vertex group is non-pinching.  


\end{itemize}
\end{definition}

A (restricted) limit group $L$ admits a finite family of well-separated resolutions, called the {\em taut Makanin-Razborov diagram of $L$}, such that any (restricted) morphism from $L$ to $\F$ factors through a well-separated resolution and moreover it is compatible with its structure (see \cite[Proposition 2.5]{Sela4}).

We next refine the class of taut morphisms by those that are in {\em shortest form}.

\begin{definition}[Shortest form morphism]\label{Minimal}
Let $Res$ be a well-separated resolution of a limit group $L:=L_0$: 
$$ L_0\rightarrow L_1\rightarrow L_2\rightarrow \ldots \rightarrow L_p$$

For each $i\leq p$, let $Comp(L_i)$ be the completion up to the step $i$ (see \cite[Definition 1.12]{MR1972179}), and  consider the multi-graded Makanin-Razborov diagram (see \cite[Section 12]{MR1863735}) of $Comp(L_{i+1})$ with respect to the images of the rigid subgroups and edge groups of the JSJ decompositions of each freely indecomposable free factor of $L_i$. 

For each morphism $h:L\rightarrow\mathbb{F}$  that factors through $Res$, there exists a sequence of morphisms $h_0, h_1, \ldots, h_p$ from $Comp(L_i)$ to $\mathbb{F}$ which are defined as sections with respect to the canonical epimorphisms $\pi_i:Comp(L)\rightarrow Comp(L_i)$, i.e. $h=h_i \circ \pi_i$, $1\leq i\leq p$. 

We say that a taut morphism, $h:L\rightarrow \mathbb{F}$, with respect to the given resolution, is in shortest form with respect to $Res$, if  every $h_i$, $1\leq i\leq p$, as above is shortest with respect to the modular automorphism of the solid (or rigid) limit group it factors out from, in the multi-graded Makanin-Razborov diagram it factors through.   
\end{definition}

Note that any morphism of a given (restricted) limit group can be extended to a (taut)  shortest form morphism that factors through a completed resolution in the taut Makanin-Razborov diagram. 

\begin{figure}[ht!]
\centering
\includegraphics[width=1\textwidth]{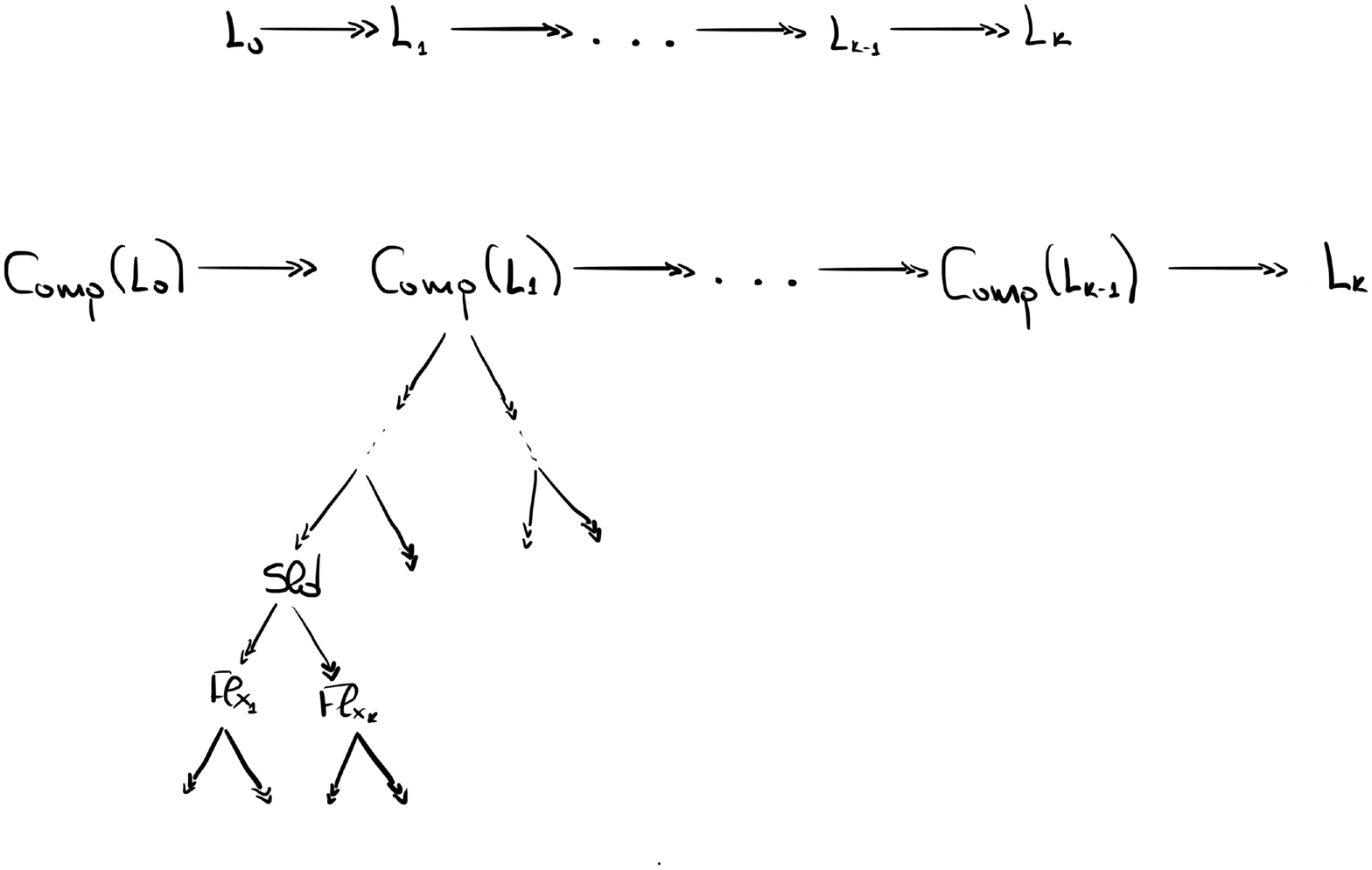}
\caption{A resolution \& a multi-graded Makanin-Razborov diagram of a completed limit group.}
\label{D}
\end{figure}

Having defined taut shortest form morphisms we are now ready to briefly describe the process of the validation of a $\forall\exists$ first-order sentence which is true in a nonabelian free group. 

We perform the same steps (1)-(3) that were described for the minimal rank case. In item (3) we only consider taut well structured resolutions.
 
In items (5) and (6) we now require the values for the variables $\bar z, \bar s$ to be images of (taut) shortest form morphisms for the fixed resolution we work with. 
 
Item (6) differs in a yet another way, it may happen that the natural image of $L_i$ in $QL_i$ is a proper quotient of $L_i$. In this case we go back to item (3) and work with this proper quotient instead of $L_i$. Otherwise we go to item (7).
 
For item (7) the construction of towers and well-separated resolutions for each $QL_i$ is much more complicated than in the minimal rank case. It is described, for example, in \cite[Section 4]{Sela4}.  Here is where we essentially use that specializations of $\bar z, \bar s$ are shortest form specializations (see \cite[Proposition 4.3]{Sela4}). For the purposes of this paper it is enough to know that this process terminates after finitely many steps (see \cite[Theorem 4.12]{Sela4}, \cite[Theorem 3.5]{KhSkBook}), the precise construction is not essential.

\section{Main Theorem}

We are now ready to prove: 

\begin{theorem}\label{AE1}
Let $d<1/16$. Then any $\forall\exists$-sentence (with coefficients) that is true in a nonabelian free group $\mathbb{F}$ is true with overwhelming probability in a random group of density $d$.
\end{theorem}
\begin{proof}
Allowing constants, by Lemma \ref{reduction}, a $\forall\exists$-sentence $\tau$ is equivalent in $\mathbb{F}$ to a sentence of the form $\forall\bar{x}\exists\bar{y}\bigl(\sigma(\bar{x},\bar{y},\bar a)=1\land \psi(\bar{x},\bar{y},\bar a)\neq 1\bigr)$, where $\sigma(\bar{x},\bar{y},\bar a)$ and   $\psi(\bar{x},\bar{y},\bar a)$ are single words in $\langle\bar{x},\bar{y}, \bar a\rangle$. Suppose $\tau$ is true in $\F$. We claim that the iterative procedure presented in Section \ref{Validation}, is also valid for a random group (of density $d$). Hence, the formal solutions produced along the procedure will validate the sentence with overwhelming probability in the corresponding definable sets in a random group.  

By Theorem \ref{Merz} there exists a formal solution $\bar{w}(\bar{x},\bar a)$ such that $\sigma(\bar{x},\bar{w}(\bar{x}, \bar a),\bar a)$ is tivial in $\langle\bar{x}, \bar{a}\rangle$ and $\bar{w}(\bar{x}, \bar a)$ witnesses that the inequation also holds in $\mathbb{F}$ for tuples in $V^c:=\mathbb{F}^{\abs{\bar{x}}}\setminus V$, where $V$ is the solution set of $\psi(\bar{x},\bar{w}(\bar{x},\bar a), \bar a)=1$ in $\mathbb{F}$, which is a proper subset of $\F^{|\bar x|}$. We claim that for any tuple $\bar{b}_\ell$ in a random group $\Gamma_\ell$ with the property that all of its pre-images are in $V^c$, then with overwhelming probability  $\sigma(\bar{b}_\ell,\bar{w}(\bar{b}_\ell,\bar a),\bar a)=1 \land \psi(\bar{b}_\ell,\bar{w}(\bar{b}_\ell,\bar a),\bar a)\neq 1$. Indeed, obviously, since $\sigma(\bar{x},\bar{w}(\bar{x},\bar a),\bar a)$ is trivial in $\langle\bar{x},\bar{a}\rangle$, we get $\sigma(\bar{b}_\ell,\bar{w}(\bar{b}_\ell,\bar a),\bar a)=1$ for any $\bar{b}_\ell$ in a random group $\Gamma_\ell$ (and actually in any group). For the inequation, 
by our hypothesis we have  $\psi(\bar{c}_\ell,\bar{w}(\bar{c}_\ell,\bar a),\bar a)\neq 1$ in $\F$ for any pre-image $\bar c_\ell$ of $\bar{b}_\ell$. Hence, by Theorem \ref{univ}, either the probability that $\psi(\bar{b}_\ell,\bar{w}(\bar{b}_\ell,\bar a),\bar a)=1$ goes to $0$ as $\ell$ goes to $\infty$ and, consequently, the probability that $\psi(\bar{b}_\ell,\bar{w}(\bar{b}_\ell,\bar a),\bar a)\neq 1$ goes to $1$, or $\bar b_\ell$ is the image of a tuple $\bar c_\ell$ that satisfies $\psi(\bar{x},\bar{w}(\bar x,\bar a),\bar a)=1$ in $\F$, which is a contradiction to the choice of $\bar b_{\ell}$. Therefore, only the first alternative holds and the result follows. 

We continue with the tuples in a random group whose pre-images not all belong to $V^c$. 

Recall, that the variety $V$, can 
be further seen as the union of irreducible varieties, i.e. varieties that correspond to limit groups. We continue with one such (restricted) limit group, say $L(\bar x, \bar a)$, and we apply Theorem \ref{MerzTowers} for $L(\bar x, \bar a)$ and a tower $T(\bar{x},\bar z, \bar a)$ in order to obtain finitely many formal solutions, $\bar{w}_1(\bar{x},\bar{z}, \bar{s},\bar a), \ldots, \bar{w}_q(\bar{x},\bar{z}, \bar{s},\bar a)$, and a set of covering closures, $C_1(\bar x, \bar z, \bar s, \bar a), \ldots, C_q(\bar x, \bar z, \bar s, \bar a)$, of $T$,  so that $\sigma(\bar{x},\bar{w}_i(\bar{x},\bar{z}, \bar{s},\bar a), \bar a)$ is trivial in $C_i$ for any $i\leq q$, i.e. for any (restricted) morphism $h:C_i\rightarrow \mathbb{F}$, $h$ kills $\sigma(\bar{x},\bar{w}_i(\bar{x},\bar{z}, \bar{s},\bar a), \bar a)$. Moreover, for each $1\leq i \leq q$ there exists a (restricted) morphism $h_i:C_i\rightarrow \mathbb{F}$ such that $\psi(h_i(\bar{x}),h_i(\bar{w}_i), \bar a)\neq 1$ in $\F$. 

At this point we split the proof in two cases. For clarity we will first prove the result in the minimal rank case. Recall that in this case the termination of the validating process is easier and the choices for the stratification of $\F^{|\bar x|}$ are less .  
\ \\ \\
{\bf The minimal rank case:} For each $i\leq q$, we consider the subset $D_{L,T}^i$ of $\mathbb{F}^{\abs{x}}$ that contains all tuples with the following properties: they extend to (restricted) morphisms from $C_i$  to $\mathbb{F}$ and any such extension {\bf does not} satisfy $\psi(\bar{x},\bar{w}_i(\bar{x},\bar{z}, \bar{s}), \bar a)=1$ in $\F$. 

$$D_{L,T}^i(\bar x):=\exists \bar z, \bar s \bigl(\Sigma_{C_i}(\bar x, \bar z, \bar s, \bar a)=1\bigr)\bigwedge \forall \bar z, \bar s\bigl(\Sigma_{C_i}(\bar x, \bar z, \bar s, \bar a)=1\rightarrow \psi(\bar{x},\bar{w}_i(\bar{x},\bar{z}, \bar{s}), \bar a)\neq 1\bigr)$$

We note in passing that the solution set of the above formula in $\F$ is a proper subset (in principle) of the ``good tuples", i.e. the tuples for which the formal solution $\bar w_i$ validates the sentence at this stage.   

By \cite[Proposition 1.2]{Sela4}, there exists a finite number of proper quotients $Q_1^i(\bar x, \bar z, \bar s, \bar a),\ldots ,$ $Q_n^i(\bar x, \bar z, \bar s, \bar a)$ of $C_i$ such that $\bar c$ belongs to $D_{L,T}^i$ if and only if it extends to a (restricted morphism from $C_i$ to $\F$ and no such (restricted) extension factors through one of the quotients $Q_1^i,\ldots , Q_n^i$. 
 
For any fixed $\ell$, we take a tuple $\bar b_\ell$ in a random group $\Gamma_\ell$ such that it has a pre-image $\bar c_\ell$ that extends to satisfy $\Sigma_{C_i}(\bar x, \bar z, \bar s, \bar a)=1$ in $\F$ and any such pre-image (that extends to satisfy $\Sigma_{C_i}(\bar x, \bar z, \bar s, \bar a)=1$)  belongs to $D^i_{L,T}$. 


We will show that for such $\bar b_\ell$ we can find a tuple $\bar d_\ell$ so that $\sigma(\bar{b_\ell},\bar d_\ell, \bar a)=1 \land \psi(\bar{b}_\ell,\bar d_\ell, \bar a)\neq 1$ with probability tending to $1$ as $\ell$ goes to $\infty$. 

Denote the formula that collects the tuples corresponding to $\bigcup\limits_{j=1}^n Hom_\F(Q_j^i,\F)$ by $\Phi_i (\bar{x},\bar{z},\bar{s}, \bar a)$. The solution set of $\Phi_i$ in $\F$ is a union of varieties and, since we work over parameters, it is equivalent to a variety that, by abusing notation, we denote by $\Phi_i(\bar{x},\bar{z},\bar{s}, \bar a)=1$. Note that, since $\bar b_\ell$ has a pre-image that extends to satisfy $\Sigma_{C_i}(\bar x, \bar z, \bar s, \bar a)=1$, 
$\bar b_\ell$ itself extends to satisfy $\Sigma_{C_i}(\bar x, \bar z, \bar s, \bar a)=1$, say by $\bar b_\ell,\bar{b}_\ell^1,\bar{b}_\ell^2$. To see this just consider the projection of the solution of $\Sigma_{C_i}(\bar x, \bar z, \bar s, \bar a)=1$ in $\F$ to $\Gamma_\ell$. By Theorem \ref{univ}, for the triple (of tuples) $\bar b_\ell, \bar{b}_\ell^1,\bar{b}_\ell^2$, either the probability that  $\Phi_i(\bar{b}_\ell,\bar{b}^1_\ell,\bar{b}^2_\ell, \bar a)= 1$ tends to $0$, as $\ell$ goes to $\infty$, or each such triple  is the image of a solution $\bar{c}_\ell,\bar{c}^1_\ell, \bar{c}^2_\ell$ of the same formula in $\F$. 
\begin{itemize}
\item In the first alternative, equivalently,   $\Phi_i(\bar{b}_\ell,\bar{b}^1_\ell,\bar{b}^2_\ell, \bar a)\neq 1$ with overwhelming probability, and we are done, since 
$$\F\models\forall \bar x,\bar z,\bar s\bigl(\Sigma_{C_i}(\bar x,\bar z,\bar s,\bar a)=1\rightarrow (\Phi_i(\bar x,\bar z,\bar s,\bar a)=1\leftrightarrow \psi(\bar x,\bar w_i(\bar x, z,\bar s,\bar a),\bar a)=1)\bigr) $$
Thus, since $\Sigma_{C_i}(\bar{b}_\ell,\bar{b}^1_\ell,\bar{b}^2_\ell, \bar a)=1$, we also get that $\psi(\bar b_{\ell}, w_i(\bar{b}_\ell, \bar{b}_\ell^1, \bar b_\ell^2, \bar a), \bar a)\neq 1$ with overwhelming probability. 

\item In the second alternative, i.e. when the tuple $\bar b_\ell, \bar{b}^1_\ell, \bar{b}^2_\ell$ is the image of a solution $\bar{c}_\ell,\bar{c}^1_\ell, \bar{c}^2_\ell$ of the same formula in $\F$, by the choice of $\bar b_\ell$ and since a fortiori $\bar{c}_\ell,\bar{c}^1_\ell, \bar{c}^2_\ell$ satisfies $\Sigma_{C_i}(\bar x, \bar z, \bar s, \bar a)=1$ in $\F$ (as the solution set of $\Phi_i(\bar{x},\bar{z},\bar{s}, \bar a)=1$ is a subset of the solution set of $\Sigma_{C_i}(\bar x, \bar z, \bar s, \bar a)=1$) we get a contradiction.  

\end{itemize}

We recall that $\psi(\bar{x},w(\bar{x},\bar{a}),\bar{a})=1$ is the union of finitely many sets of the form $Hom_{\F}(L_i(\bar{x},$ $\bar{a}),\F)$, for $i\leq p$, and for each $i\leq p$, there exist finitely may towers $T_{i,j}(\bar{x},\bar{z}, \bar{a})$, for $j\leq m_i$, such that any (restricted) morphism from $L_i$ to $\F$ extends to a morphism from $T_{i,j}$ to $\F$, for some $j\leq m_i$. Moreover, for each $T_{i,j}$, for $i\leq p$, and $j\leq m_i$ there exist finitely many towers $C_{i,j,r}$, for $r\leq q_{i,j}$, that form a covering closure of $T_{i,j}$. 

At the end of this second step we have validated all tuples $\bar b_\ell$ in a random  group $\Gamma_\ell$ that have a pre-image that extends to some $C_{i,j,r}$, for $i\leq p$, $j\leq m_i$ and $r\leq q_{i,j}$, and any such pre-image belongs to $D^r_{L_i, T_{i,j}}$.

The validation process in a nonabelian free group, continues with each $QL^{i,j,r}_k$, which is a quotient of $C_{i,j,r}$ that collects ``bad tuples", in parallel. For each such quotient there exist finitely many towers (of lower complexity than $C_{i,j,r}$) and for each such tower a set of towers that form a  covering closure and bear formal solutions. At this step we validate all $\bar b_\ell$ that extend to one of this closures and any of such extensions does not satisfy $\psi(\bar x, \bar y, \bar a)=1$ with the new formal solution corresponding to the particular closure at the place of $\bar y$. This process produces a stratification of $\F^{|\bar x|}$ by finitely many subsets $Str_1, Str_2 , \ldots, Str_\alpha$, such that $Str_i$ collects the tuples that work at step $i$.


We finally fix a $\forall\exists$ sentence, say $\tau$, which is true in $\F$. For $\ell\in \mathbb{N}$, consider a group $\Gamma_{\ell}$ that does not satisfy $\tau$. Suppose $\bar b_\ell$ witnesses it. As $\ell$ goes to $\infty$, the witnesses $\bar b_\ell$ will have pre-images that vary through the finitely many strata, but, for any stratum, the ratio of the number of groups for which the sentence will be false over the number of all groups goes to $0$, hence, since there are only finitely many strata,  $\tau$ is true in a random group with overwhelming probability.     




\ \\ \\
{\bf The general case} In this case, we start with the restricted limit group $L(\bar x, \bar a)$ and a tower $T(\bar x, \bar z, \bar a)$ that corresponds to a well-separated resolution of $L$. We apply Theorem \ref{MerzTowers} for $L$ and $T$ in order to obtain finitely many formal solutions $\bar w_i(\bar x, \bar z, \bar s, \bar a)$, for $i\leq q$, and a set of covering closures $C_i(\bar x, \bar z, \bar s, \bar a)$, for $i\leq q$, of $T$, so that $\sigma(\bar x, \bar w_i(\bar x, \bar z, \bar s, \bar a), \bar a)$ is trivial in $C_i$, for any $i\leq q$, and, moreover, there exists a restricted morphism $h_i:C_i\rightarrow \F$ with  $\psi(h_i(\bar x), h_i(\bar w_i), \bar a)\neq 1$.

We fix some $C_i$, for $i\leq q$. Let $S_1, \ldots , S_k$ be a collection of solid (or rigid) limit groups that appear in the multi-graded Makanin-Razborov diagrams of different levels of its resolution (see Definition \ref{Minimal}). 

Recall that any morphism $h:L\rightarrow \F$, extends to some $C_i$ as (taut) shortest form with respect to some collection of solid limit groups.   

We consider the subset $C^i_{S_1,\ldots , S_k}$ of $\mathbb{F}^{\abs{\bar x}}$ that contains all tuples with the following properties:
\begin{itemize}
    \item they extend to taut morphisms that factor through the fixed $C_i$, and factor out, as solid (or rigid) morphisms, through the fixed $S_i$, for $i\leq k$; 
   
    \item no shortest form, with respect to $Mod(S_1),\ldots , Mod(S_k)$, such extension satisfies $\psi(\bar{x},\bar{w}_i(\bar{x},$ $\bar{z}, \bar{s}, \bar{a}), \bar{a})=1$.
\end{itemize} 





We consider tuples $\bar b_\ell$ in the random group $\Gamma_\ell$ with the following properties:
 \begin{enumerate}
     \item $\bar b_\ell$ has a pre-image $\bar c_\ell$, that admits a (taut)  shortest form extension, with respect to $Mod(S_1),$ $\ldots, Mod(S_k)$, factoring through $C_i$;
     \item any such pre-image of $\bar b_{\ell}$ belongs to $C^i_{S_1,\ldots,S_k}$;
     \item no pre-image of $\bar b_{\ell}$ extends to factor through some well-separated resolution that imposes a more refined pinching than (the well-separated resolution that corresponds to) $C_i$. 
 \end{enumerate}
 
The last point guarantees that any pre-image of $\bar b_{\ell}$ that extends to a morphism from $C_i$ is taut with respect to the fixed resolution. 

 
 We will show that  such $\bar b_\ell$ can be extended to $\bar{b}_\ell,\bar{b}^1_\ell,\bar{b}^2_\ell$ so that $\sigma(\bar{b}_\ell,\bar{w}_i(\bar{b}_\ell,\bar{b}^1_\ell, \bar b^2_\ell), \bar a)=1$ and $\psi(\bar{b}_\ell,\bar{w}_i(\bar{b}_\ell,\bar{b}^1_\ell, \bar b^2_\ell), \bar a)\neq 1$ with overwhelming probability.

 Indeed, since $\bar b_\ell$ has a pre-image $\bar c_\ell$ that extends to a 
 morphism that factors through $C_i$ and does not satisfy $\psi(\bar{x},\bar{w}_i(\bar{x},\bar{z},\bar{s}, \bar a), \bar a)=1$, we consider the image, say $\bar b_\ell, \bar b_\ell^1, \bar b_\ell^2$, of that triple (of tuples). It is obvious that it satisfies $\Sigma_{C_i}(\bar x, \bar z, \bar s, \bar a)=1$ with overwhelming probability. We next consider the conjunction of $\psi(\bar{x},\bar{w}_i(\bar x, \alpha(\bar{z}), \alpha(\bar{s}),\bar a),\bar a)=1$ for $\alpha$ in $Mod(S_1)\cdot$ 
$\ldots\cdot Mod (S_k)$. Since $\mathbb{F}$ is equationally Noetherian the above (infinite) system of equations is equivalent to a finite subsystem. Denote this finite subsystem by $\Phi^{Mod}_i (\bar{x},\bar{z},\bar{s}, \bar a)=1$, we further add to $\Phi^{Mod}_i(\bar{x},\bar{z},\bar{s}, \bar a)=1$ the equations $\Sigma_{C_i}(\bar x, \bar z, \bar s, \bar a)=1$ and the equations for the fixed solid (or rigid) limit  groups $\Sigma_{S_1}(\bar{x},\bar{z},\bar{s}, \bar a)=1,\ldots,\Sigma_{S_k}(\bar{x},\bar{z},\bar{s}, \bar a)=1$. We still denote this system of equations by $\Phi_i^{Mod}(\bar{x},\bar{z},\bar{s}, \bar a)=1$.



By Theorem \ref{univ},  either the probability that  $\Phi^{Mod}_i(\bar{b}_\ell,\bar{b}^1_\ell, \bar{b}^2_\ell,\bar a)=1 $ is zero
or $\bar{b}_\ell,\bar{b}^1_\ell, \bar{b}^2_\ell$ is the image of  $\bar{c}_\ell,\bar{c}^1_\ell, \bar{c}^2_\ell$ such that $\Phi^{Mod}_i(\bar{c},\bar{c}_1, \bar{c}_2, \bar a)=1 $ in $\F$. 

If the probability that $\Phi^{Mod}_i(\bar{b}_\ell,\bar{b}^1_\ell, \bar{b}^2_\ell, \bar a)=1 $ is zero, then, equivalently, $\Phi^{Mod}_i(\bar{b}_\ell,\bar{b}^1_\ell, \bar{b}^2_\ell, \bar a)\neq 1$ with overwhelming probability. Thus, since $\bar{b}_\ell,\bar{b}^1_\ell, \bar{b}^2_\ell$ satisfies $\Sigma_{C_i}(\bar{x},\bar{z},\bar{s}, \bar a)=1, \Sigma_{S_1}(\bar{x},\bar{z},\bar{s}, $ $\bar a)=1, \ldots, \Sigma_{S_k}(\bar{x},\bar{z},\bar{s}, \bar a)=1$, we get that  $\psi(\bar{b}_\ell,  \bar w_i(\bar b_\ell, \alpha (\bar{b}^1_\ell), \alpha (\bar{b}^2_\ell), \bar a), \bar a)\neq 1$ for some $\alpha\in Mod(S_1)\cdot\ldots\cdot Mod (S_k)$  with overwhelming probability and $\bar w_i(\bar b_\ell, \alpha (\bar b_\ell^1), \alpha (\bar b_\ell^2), \bar a)$ validates the sentence for $\bar b_\ell$.  

If the tuple $\bar b_\ell, \bar{b}^1_\ell, \bar{b}^2_\ell$  is the image of a solution $\bar{c},\bar{c}_1, \bar{c}_2$ such that $\Phi^{Mod}_i(\bar{c},\bar{c}_1, \bar{c}_2)=1 $ in $\F$, then, $\bar{c},\bar{c}_1, \bar{c}_2$ satisfies $\Sigma_{C_i}(\bar x, \bar z, \bar s, \bar a)=1$, and $\Sigma_{S_1}(\bar{x},\bar{z},\bar{s},\bar a)=1, \ldots, \Sigma_{S_k}(\bar{x},\bar{z},\bar{s},\bar a)=1$  in addition, we may change $\bar c_1$, $\bar c_2$ by an automorphism $\alpha$ to make them shortest form since they satisfy $\Phi_i(\bar{x}, \bar{w}_i(\bar{x}, \alpha (\bar{z}), \alpha (\bar{s}),\bar a)),\bar a)=1$ for any $\alpha$ in $Mod(S_1)\cdot$ $\ldots\cdot Mod (S_n)$. Therefore we assume  that $\bar c, \bar{c}_1, \bar{c}_2$ is shortest form. 
 Hence, this solution, by the choice of $\bar b_\ell$, must belong to $C^i_{S_1, \ldots, S_k}$ and in particular it factors out  as solid morphisms from $S_1, \ldots, S_k$, and does not satisfy $\psi(\bar x, \bar w_i(\bar x, \bar z, \bar s, \bar a), \bar a)=1$, a contradiction.

Those  tuples $\bar b$  that have  pre-images that extent to $\bar{c}_1,\bar{c}_2$ satisfying equations of $S_1,\ldots ,S_k$ that are not in $C_{S_1,\ldots , S_k}$  either correspond to morphisms that factor out from a different set of solid (or rigid) limit groups, $S_{i_1},\ldots , S_{i_k}$, in the multi-graded Makanin-Razborov diagrams of the different levels of the completed resolution, or have pre-images $\bar c$ that extend to $\bar c_1, \bar c_2$  which are shortest form with respect to $Mod(S_1)\ldots Mod(S_n)$ and also satisfy $\psi(\bar{x},\bar{w}_i(\bar{x},\bar{z}, \bar{s},\bar a),\bar a)=1$. In the first case, we tackle the finitely many cases in parallel, at the same step. In the second case, we define quotients of each $C_i$ with which we will proceed to the next step and will repeat similar constructions.    

Let $Q^i_1(\bar x, \bar z, \bar s, \bar a), \ldots,Q^i_{p_i}(\bar x, \bar z, \bar s, \bar a) $ be the quotient groups of $C_i(\bar x, \bar z, \bar s, \bar a)$ obtained by all (taut) shortest form morphisms, with respect to $Mod(S_1), \ldots ,Mod(S_k)$, that factor through the fixed resolution and in addition satisfy $\psi(\bar{x},\bar{w}_i(\bar{x},\bar{z}, \bar{s},\bar a),\bar a)=1$. We assume that the natural image of $L(\bar x, \bar a)$ 
in $Q^i_j(\bar x, \bar z, \bar s, \bar a)$, for $j\leq p_i$, is not a proper quotient of $L$, otherwise we repeat this step with the image of $L$.   
Each $C_i$ and every set of solid (or rigid) limit groups in the multi-graded Makanin-Razborov diagrams from the different levels of the completed resolution   are considered independently. At this step we verified the sentence simultaneously for all tuples $\bar b_\ell$ satisfying the three conditions for some $i\leq q$ and some set, $S_1, \ldots, S_k$, of solid (or rigid) limit groups.  

Remaining tuples $\bar b_\ell$ either have pre-images $\bar c_{\ell}$ that extend to $\bar c_1,\bar c_2$ that factor through one of $Q^i_1(\bar x, \bar z, \bar s, \bar a), \ldots,Q^i_k(\bar x, \bar z, \bar s, \bar a) $ or  have pre-images $\bar c_\ell$ that extent to $\bar c_3,\bar c_4$ that factor through another $C_j$ that imposes more refined pinching than $C_i$.  In other words these $\bar b_\ell$ have pre-images $\bar c_\ell$ that belong to a different stratum in the process for $\mathbb F$. 

We now give the final argument of the proof. Suppose $\tau$ is a $\forall\exists$ sentence, which is true for a nonabelian free group.   The main point is that the validating process decomposes $\F^{| \bar x|}$ in finitely many different strata. For $\ell\in \mathbb{N}$, consider a group $\Gamma_{\ell}$ that does not satisfy $\tau$. Suppose $\bar b_\ell$ witnesses it. As $\ell$ goes to $\infty$, the witnesses $\bar b_\ell$ will have pre-images that vary through the finitely many strata, but, for any stratum, the ratio of the number of groups for which the sentence will be false over the number of all groups goes to $0$, hence, since there are only finitely many strata,  $\tau$ is true in a random group with overwhelming probability.
\end{proof}

\begin{theorem}
Let $d<1/16$. Every $\exists\forall$-sentence that is true in a nonabelian free group, it is true in the random group.
\end{theorem}

\begin{proof} 
Let $Q(\bar x, \bar y)$ be a quantifier free formula such that 
$$\F\models \exists \bar x \forall \bar y Q(\bar x, \bar y)$$ 
for a nonabelian free group $\F$. By \cite{MR2238945}, \cite{MR2293770} the sentence is true in the nonabelian free group $\F_n$ of rank $n$. Hence, for some $\bar b$ in $\F_n$, we have $\forall \bar yQ(\bar b, \bar y)$ is true in $\F_n$. Thus, by Theorem \ref{UnivwithCon1}, it is true with overwhelming probability in a random group of density $d$. Therefore, $\exists \bar x \forall \bar y Q(\bar x, \bar y)$ is true with overwhelming probability in a random group of density $d$.
%
 \end{proof}

The last two theorems imply the following.

\begin{corollary} 
Let $d<1/16$. A $\forall\exists$-sentence is true in a nonabelian free group  if and only if it is  true with overwhelming probability in a random group of density $d$.
\end{corollary}

\bibliography{biblio}

\end{document}